\documentclass[12pt]{article}
\usepackage{amsmath,amsfonts, amsthm}
\usepackage{verbatim}
\usepackage{relsize}
\usepackage{enumerate}
\normalbaselines
\newtheorem{teor}{Theorem}

\newtheorem{lema}[teor]{Lemma}
\newtheorem{coro}[teor]{Corollary}
\newtheorem{rem}[teor]{Remark}

\newtheorem{defi}[teor]{Definition}

\title{Rigidity results for stochastically 
complete maximal hypersurfaces in
Generalized Robertson-Walker spacetimes}

\author{Mar\'ia \'A. Medina and Jos\'e A. S. Pelegr\'in \\[6mm]
Departamento de Inform\'atica y An\'alisis Num\'erico, \\[0.5mm]
Universidad de C\'ordoba, 14071 C\'ordoba, Spain \\ E-mails\textup{:
\texttt{medipoz@gmail.com, jpelegrin@uco.es}} \\[3mm]}

\date{}

\begin{document}

\maketitle

\thispagestyle{empty}

\begin{abstract}
In this article we obtain new rigidity results 
for stochastically complete maximal 
hypersurfaces in Generalized Robertson-Walker spacetimes that 
satisfy the Null Energy Condition. Under 
appropiate geometric assumptions we prove new parametric uniqueness and 
nonexistence results as 
well as obtain a Calabi-Bernstein type result for the maximal 
hypersurface equation in these ambient spacetimes.
\end{abstract}
\vspace*{5mm}



\section{Introduction}

A maximal hypersurface in a Lorentzian manifold is a spacelike 
hypersurface with zero mean curvature. Their name is due to 
the fact that they represent critical points of the area functional 
for compactly supported normal variations \cite{BF} that 
are area maximizing in certain ambient manifolds \cite{CPR}. Their 
importance in General Relativity comes from the fact that they describe 
the transition between an expanding and a contracting phase 
of the universe in some relevant models. They also  
enable us to understand the geometric structure of the spacetime 
since in some 
asymptotically flat spacetimes there exists a foliation by 
maximal hypersurfaces \cite{BF}. Moreover,  
the Cauchy problem for Einstein's equation with initial 
conditions on a maximal hypersurface 
is reduced to a second order nonlinear elliptic differential 
equation and a first order linear differential system \cite{Li}.

From a mathematical perspective, maximal hypersurfaces 
present interesting 
uniqueness properties. Indeed, the classical Calabi-Bernstein theorem 
states that the only complete maximal hypersurfaces in the 
Lorentz-Minkowski 
spacetime $\mathbb{L}^{n+1}$ are the spacelike 
hyperplanes. This result was proved 
by Calabi \cite{Ca} for $n\leq 4$ and extended to 
arbitrary dimension by Cheng 
and Yau \cite{CY}. These uniqueness results have been 
generalized to a wide variety of 
ambient spacetimes (see, for instance, 
\cite{A-R-S1, BF, Mo, PRR1, PRR3}). Among these spacetimes, 
our results will be focused on 
the models known as 
Generalized Robertson-Walker (GRW) spacetimes, introduced 
in \cite{A-R-S1} as an extension of the classical 
Robertson-Walker models to the case 
where the fiber does not necessarily have constant sectional 
curvature. As a consequence, GRW spacetimes can describe 
the universe in a more 
accurate scale where it may not be spatially 
homogeneous \cite{RaS}. 

The first 
attempts to obtain uniqueness results for maximal 
hypersurfaces in GRW spacetimes  
relied on the compactness of the 
hypersurface \cite{A-R-S1}. Since then, many efforts 
have been made to extend these results to the 
non compact case due to the fact that 
the experimental data and certain theoretical arguments 
suggest that spatially open 
models (i.e., spacetimes that do not admit any compact 
spacelike hypersurface) provide a more accurate 
description of our current universe 
\cite{Chiu}. For instance, in \cite{RRS1} and 
\cite{RRS2} the compactness 
assumption was replaced by assuming the parabolicity 
of the maximal hypersurface. Nevertheless, parabolic 
spacelike hypersurfaces 
might lead to a violation of the holographic principle 
under the spacelike 
entropy bound \cite{P}, encouraging the use of different 
assumptions such 
as certain energy conditions \cite{PRR2} or being contained in a slab 
\cite{RRS3} to obtain new Calabi-Bernstein type
results for complete maximal hypersurfaces in GRW spacetimes.

A crucial assumption in all the 
aforementioned results was the 
completeness of the maximal hypersurface, which allowed 
the use of certain 
techniques such as some integral inequalities for the 
compact case,  
maximum principles for parabolic Riemannian manifolds or the 
Omori-Yau maximum principle. In this article we will study 
stochastically complete maximal hypersurfaces, which may 
not be complete as we 
will see in Section \ref{s2}. Indeed, stochastic completeness 
is a more general concept 
than parabolicity, since any parabolic Riemannian 
manifold is 
stochastically complete \cite[Cor. 6.4]{Gr}. Moreover, this 
property has an interesting 
physical meaning. Namely, if a particle follows a Brownian 
motion on a 
stochastically complete Riemannian manifold the total probability of 
finding it in the state space is constantly equal 
to $1$, i.e., stochastic processes 
on these manifolds have infinite lifetime. 

Our aim in this article is to obtain new rigidity 
results for stochastically complete maximal hypersurfaces 
in GRW spacetimes. The paper is organized as 
follows. The fundamental definitions concerning these 
geometric objects as well as a deeper insight in the 
notion of stochastic completeness will be given in 
Section \ref{s2}. Our main parametric results will 
be obtained in Section \ref{sesu}, where a relation between 
the height function of a stochastically complete 
maximal hypersurface and the warping 
function of the GRW spacetime is obtained in Theorem 
\ref{teotau}. We recall that previous works for stochastically 
complete spacelike hypersurfaces 
in GRW spacetimes have been focused on finding 
certain bounds for 
the height function in terms of the mean curvature 
\cite{ARSc}. These bounds on 
the height function have been combined 
with more restrictive 
energy assumptions and hypotheses on the norm of the 
second fundamental form to obtain some uniqueness 
results \cite{Lim, LP}. However, to prove our main results 
in this section we will not impose 
any bound on the maximal hypersurface's height function. Instead, 
we will deal with stochastically complete maximal 
hypersurfaces with bounded hyperbolic angle to obtain our main 
uniqueness result (Theorem \ref{teomon}) in GRW spacetimes 
that obey the 
Null Convergence Condition (which is a 
weaker energy condition than the one used in 
\cite{ARR, A-R-S1, Lim}) as 
well as in the special case of product 
spacetimes (Theorem \ref{teopro}). Moreover, we will also obtain 
interesting corollaries for well known models of the universe 
as well as new non existence results for these type of 
hypersurfaces (Theorem \ref{teononex}). In addition, in 
Section \ref{suselm} we study stochastically complete maximal 
hypersurfaces in Lorentz-Minkowski and de Sitter spacetimes and 
extend the classical Calabi-Bernstein theorem to these type 
of hypersurfaces in Theorem \ref{teosf}. To conclude, 
Section \ref{secb} is devoted to use our previous 
parametric theorems to obtain a new uniqueness 
result for maximal graphs in GRW spacetimes (Theorem \ref{teocb}).

\section{Preliminaries}
\label{s2} 

\subsection{Maximal hypersurfaces in GRW spacetimes}

Let $(F,g_F)$ be an $n(\geq 2)$-dimensional (connected) Riemannian 
manifold, $I \subseteq \mathbb{R}$ an open interval 
and $f$ a positive smooth function on $I$. Consider 
the product manifold 
$I \times F$ endowed with the Lorentzian metric

\begin{equation}
\label{metrica}
\overline{g} = -\pi^*_{I} (dt^2) +f(\pi_{I})^2 \, \pi_{F}^* (g_F), 
\end{equation}

\noindent where $\pi_{I}$ and $\pi_{F}$ denote the projections onto $I$ and
$F$, respectively. The resulting Lorentzian manifold $(I \times F, \overline{g})$ is 
a warped product with base $(I,-dt^2)$, 
fiber $(F,g_F)$ and warping function $f$ \cite[Chap. 7]{O'N}. $(I \times F, \overline{g})$ 
endowed with the time orientation 
induced by $\partial_t := \partial / \partial t$ is called 
an $(n+1)$-dimensional 
Generalized Robertson-Walker (GRW) spacetime \cite{A-R-S1} and will be 
denoted by $I \times_f F$.

In any GRW spacetime the distinguished timelike and future pointing
vector field $K: = f({\pi}_I)\partial_t$ satisfies

\begin{equation}
\label{conexion} 
\overline{\nabla}_X K = f'({\pi}_I)\,X
\end{equation}

\noindent for any vector field $X$, where $\overline{\nabla}$
is the Levi-Civita connection of (\ref{metrica}). Thus, $K$ is conformal 
and its metrically equivalent $1$-form is closed.

A smooth immersion $\psi: M \longrightarrow I \times_f F$ of an 
$n$-dimensional manifold $M$ is called a spacelike hypersurface if the
Lorentzian metric (\ref{metrica}) induces a Riemannian
metric $g$ on $M$ through $\psi$. We will denote the
restriction of $\pi_I$ along $\psi$ by $\tau$. In addition, the time-orientation 
of $I \times_f F$ allows to define on 
every spacelike hypersurface $M$ in $I \times_f F$, a unique unitary
timelike vector field $N \in \mathfrak{X}^\bot(M)$ globally defined
on $M$ with the same time orientation as $\partial_t$.

This unitary normal timelike vector 
field $N$ enables us to define an 
associated shape operator $A$ and a mean curvature function 
given by $H:= -(1/n) \mathrm{trace}(A)$, where the minus sign is introduced 
to write the mean curvature vector field of 
$\psi$ as $HN$. A spacelike hypersurface with zero mean curvature is called 
a maximal hypersurface. Recall that $H \equiv 0$ if and only if 
the spacelike hypersurface is,
locally, a critical point of the $n$-dimensional area functional \cite{BF}. In addition, 
the hyperbolic 
angle $\varphi$ at any point of $M$ between $N$ and $\partial_t$ is defined by 

\begin{equation}
\label{coosh}
\cosh \varphi=-\overline{g}(N,\partial_t).
\end{equation}

Any GRW spacetime $I \times_f F$ is foliated by the distinguished family 
of spacelike hypersurfaces known as   
spacelike slices. Each spacelike slice $\{t_{0}\}\times F$, $t_{0}\in I$ is 
totally umbilical with constant
mean curvature $H=f'(t_{0})/f(t_{0})$. Hence,  a spacelike slice
is maximal (and thus, totally
geodesic) if and only if $f'(t_{0})=0$. A spacelike 
hypersurface in $I \times_f F$ is contained in a
spacelike slice if and only if the hyperbolic angle $\varphi$
identically vanishes. 

Let $\psi: M \rightarrow I \times_f F$ be an $n$-dimensional
spacelike hypersurface in a GRW spacetime. Denoting by

$$\partial_t^\top:= \partial_t+\overline{g}(N,\partial_t)N$$

\noindent the tangential component of $\partial_t$ along $\psi$, we 
see that the gradient of $\tau$ is

\begin{equation}\label{part}
\nabla \tau=-\partial_t^\top.
\end{equation}

Indeed, relating the hyperbolic angle and the norm of the 
gradient of $\tau$ by means of (\ref{coosh}) and (\ref{part}), we have

\begin{equation}
\label{se2vp}
\sinh^2 \varphi = |\nabla \tau|^2.
\end{equation}

Moreover, for the tangential
component of $K$ along $\psi$ we get from (\ref{conexion})

\begin{equation}\label{gradcosh}
\nabla \overline{g}(K,N)=-AK^\top,
\end{equation}

\noindent where $A$ is the shape operator with respect to $N$. Denoting 
by $\nabla$ the Levi-Civita
connection of the metric $g$, the Gauss and Weingarten
formulas for the immersion $\psi$ are given, respectively, by

\begin{equation}\label{GF}
\overline{\nabla}_X Y=\nabla_X Y-g(AX,Y)N
\end{equation}

\noindent and

\begin{equation}\label{WF}
AX=-\overline{\nabla}_X N,
\end{equation}

\noindent where $X,Y\in\mathfrak{X}({M})$. Combining
(\ref{conexion}) with (\ref{GF}) and (\ref{WF}) we obtain

\begin{equation}\label{KT}
\nabla_X K^\top=-f(\tau)\overline{g}(N,\partial_t)AX+f'(\tau)X,
\end{equation}

\noindent where $X\in\mathfrak{X}({M})$ and $f'(\tau):=f'\circ \tau$. Also, 
using (\ref{coosh}) and (\ref{KT}) we can compute the gradient 
of $\cosh \varphi$, obtaining

\begin{equation}
\label{gch}
\nabla \cosh \varphi=A\partial_t^\top + \frac{f'(\tau)}{f(\tau)} \cosh \varphi \ \partial_t^\top.
\end{equation}

\noindent Furthermore, (\ref{part}) and (\ref{KT}) yield

\begin{equation}
\label{nt}
\nabla_X \partial_t^\top = \frac{f'(\tau)}{f(\tau)} \ g(X, \partial_t^\top) \ \partial_t^\top + \cosh \varphi \ A X + \frac{f'(\tau)}{f(\tau)} X.
\end{equation}

These results will help us to prove our main theorems along the paper.

\subsection{Stochastic completeness}

Stochastic completeness is the property of a stochastic process to 
have infinite lifetime. In other words, if a particle follows a 
Brownian motion, the total probability of finding it in the state space 
is constantly equal to $1$. This can be 
expressed analytically as follows.

\begin{defi}
\label{defsc}
A Riemannian manifold $(M, g)$ is said to be stochastically complete 
if for some (and hence, any) $(x, t) \in M \times (0,+\infty)$

$$ \int_M p(x, y, t) dy = 1,$$

\noindent where $p(x, y, t)$ is the heat kernel of the 
Laplacian operator. 
\end{defi}

Note that in the above definition the Riemannian manifold $(M,g)$ 
is not assumed to be geodesically complete. In fact, as it was 
shown in \cite{Do}, we can construct a minimal heat kernel on an 
arbitrary Riemannian manifold as the supremum of the Dirichlet heat 
kernels on an exhausting sequence of relatively compact domains 
with smooth boundary. This analytic definition is equivalent to 
several other properties (for a proof, see \cite[Thm. 6.2]{Gr} and 
\cite[Thm. 2.8]{AMR}, for instance).

\begin{lema}
\label{lemastco}
Let $(M, g)$ be a Riemannian manifold. Then the following statements are 
equivalent:

\begin{enumerate}[(i)]
\item M is stochastically complete.

\item For every $\lambda > 0$, the only nonnegative bounded $C^2$ 
solution of $\Delta u \geq \lambda u$ on $M$ is $u \equiv 0$.

\item For every function $u \in C^2(M)$ with $u^* = \sup_M u < + \infty$ 
there exists a sequence of 
points $\{p_k \} \subset M$ satisfying 

$$ u (p_k) > u^* - \frac{1}{k} , \ \mathrm{and} \ \Delta u (p_k) < \frac{1}{k}$$

\noindent for each $k \in \mathbb{N}$.

\item For every function $f \in C^0(M)$ and every 
$u \in C^2(M)$ with $u^* = \sup_M u < + \infty$ solving the 
differential inequality $\Delta u \geq f(u),$ we have $f(u^*) \leq 0$.
\end{enumerate}

\end{lema}

As we have previously mentioned, stochastic completeness and geodesic 
completeness are independent concept. Indeed, the first 
example of a geodesically complete stochastically incomplete 
manifold was found in \cite{Az}. In his example, the 
Riemannian manifold's sectional curvature rapidly diverges to $- \infty$ 
as the Riemannian distance to a fixed point increases. In some sense 
the negative curvature plays the role of a drift that sweeps the 
Brownian particle to infinity in finite time. Conversely, since 
stochastic completeness is preserved by substracting a 
compact \cite[Cor. 6.5]{Gr}, we can easily see that 
$\mathbb{R}^3 \setminus \{0\}$ with the usual Euclidean metric 
is a geodesically incomplete stochastically complete manifold
(see \cite[Ex. 2.2]{AMR} for a different proof of this 
result). Therefore, an interesting problem is to determine which 
geometric properties ensure the stochastic completeness of a Riemannian 
manifold. For the moment, the best volume growth sufficient 
condition for stochastic completeness of a complete Riemannian manifold 
is the next one obtained in \cite{Gri}.

\begin{lema}
\label{teogri}
Let $(M,g)$ be a complete Riemannian manifold. If, for some point 
$o \in M$ the geodesic ball centered at $o$ with radius 
$r$, $B_r$, satisfies

$$\frac{r}{\log \mathrm{vol}(B_r)} \not\in L^1(+\infty),$$

\noindent then $M$ is stochastically complete.
\end{lema}

Moreover, following the terminology introduced in \cite{PRS}, 
condition \textit{(iii)} in Lemma \ref{lemastco} is also 
known as the weak maximum principle for the Laplacian. As a consequence, 
any Riemannian manifold where the Omori-Yau maximum principle holds 
will be stochastically complete. Hence, from the results obtained in 
\cite{Om} and \cite{Ya} we have

\begin{lema}
\label{teooy}
Let $(M,g)$ be a complete Riemannian 
manifold with Ricci curvature bounded 
from below. Then, the Omori-Yau maximum principle for the Laplacian 
holds on $M$. In particular, $M$ is stochastically complete.
\end{lema}

In fact, this bound on the Ricci curvature was weakened in 
\cite[Thm. 2.13]{AMR}, obtaining the maximum amount of negative 
curvature that can be allowed to preserve stochastic completeness.

\begin{lema}
\label{teoricsc}
Let $(M,g)$ be a complete $n$-dimensional Riemannian manifold, let 
$o \in M$ be a fixed origin and denote by $r$ the Riemannian distance 
function from $o$. If the radial Ricci curvature satisfies

$$\mathrm{Ric}(\nabla r, \nabla r) \geq -(n-1) G^2(r),$$

\noindent for some positive nondecreasing continuous function 
$G$ with 

$$\frac{1}{G} \not\in L^1(+\infty).$$

\noindent Then $M$ is stochastically complete.
\end{lema}

\section{Main results}
\label{sesu}

Let us begin by relating the existence of a stochastically complete 
maximal hypersurface in a GRW spacetime with bounded height 
function with the behaviour of the warping function by means 
of the next result.

\begin{teor}
\label{teotau}
Let $\psi: M \longrightarrow I \times_f F$ be a stochastically complete 
maximal hypersurface in a GRW spacetime.

\begin{enumerate}[(i)]
\item If $\tau^* = \sup_M \tau < + \infty$, then $\dfrac{f'}{f} (\tau^*) \geq 0$.

\item If $\tau_* = \inf_M \tau > - \infty$, then $\dfrac{f'}{f} (\tau_*) \leq 0$.
\end{enumerate}
\end{teor}

\begin{proof}

Using \eqref{part} and \eqref{nt} we can compute the height's function 
Laplacian, obtaining

\begin{equation}
\label{latau}
\Delta \tau = - \frac{f'(\tau)}{f(\tau)} ( n + \sinh^2 \varphi ).
\end{equation}

Hence, if $M$ is stochastically complete and 
$\tau^* = \sup_M \tau < + \infty$, \eqref{latau} and 
Lemma \ref{lemastco} yield $\dfrac{f'}{f} (\tau^*) \geq 0$. 

On the other hand, if $\tau_* = \inf_M \tau > - \infty$, 
from Lemma \ref{lemastco} it is easy to 
see that there exists a sequence 
of points $\{q_k \}_{k \in \mathbb{N}} \subset M$ such that 

$$\lim_{k \to +\infty} \tau(q_k) = \tau_* , \ \mathrm{and} \ \lim_{k \to +\infty} \Delta \tau (q_k) \geq 0 .$$

Using \eqref{latau} again we obtain $\dfrac{f'}{f} (\tau_*) \leq 0$.
\end{proof}

As a direct consequence of Theorem \ref{teotau} we obtain the following 
nonexistence result for stochastically complete maximal surfaces with 
bounded height function (compare with \cite[Prop. 4.3]{ARSc}).

\begin{coro}
\label{cortau}
Let $I \times_f F$ be a GRW spacetime.

\begin{enumerate}[(i)]
\item If $\sup (f'/f) < 0$, there are no stochastically complete 
maximal hypersurfaces bounded away from future infinity.

\item If $\inf (f'/f) > 0$, there are no stochastically complete 
maximal hypersurfaces bounded away from past infinity.
\end{enumerate}
\end{coro}

Thus, a bound on the height's function restricts the existence of 
stochastically complete maximal hypersurfaces in GRW spacetimes 
with a monotone expanding/contracting behaviour. Our 
next approach will be focused on the hyperbolic angle. Hence, 
consider a maximal hypersurface $\psi: M \longrightarrow I \times_f F$ 
and let us compute the Laplacian of the 
function $\sinh^2 \varphi$ on $M$ as follows.

\begin{equation}
\label{la1}
\Delta \sinh^2 \varphi = 2 \cosh \varphi \ \Delta \cosh \varphi + 2 |\nabla \cosh \varphi|^2.
\end{equation}

Choosing a local orthonormal reference frame $\{E_1, \dots, E_n \}$ on $TM$ we obtain the Laplacian of $\cosh \varphi$ using (\ref{gch}) 

\begin{equation}
\label{lap1}
\Delta \cosh \varphi  = \sum_{i=1}^n g(\nabla_{E_i} (A \partial_t^\top), E_i) +   \sum_{i=1}^n g \left(\nabla_{E_i} \left(\frac{f'(\tau)}{f(\tau)} \cosh \varphi \ \partial_t^\top \right), E_i\right).
\end{equation}

\noindent Moreover, we can rewrite (\ref{lap1}) as

\begin{eqnarray}
\label{lap3}
\Delta \cosh \varphi &=& \sum_{i=1}^ng((\nabla_{E_i} A) \partial_t^\top, E_i) + \sum_{i=1}^ng(\nabla_{E_i} 
\partial_t^\top, A E_i) - \frac{f''(\tau)}{f(\tau)} \cosh \varphi \sinh^2 \varphi \\ \nonumber
& &  + 2 \frac{f'(\tau)^2}{f(\tau)^2} \cosh \varphi \sinh^2 \varphi + \frac{f'(\tau)}{f(\tau)} g(A 
\partial_t^\top, \partial_t^\top) \\ \nonumber
& & + \frac{f'(\tau)}{f(\tau)} \cosh \varphi \sum_{i=1}^ng(\nabla_{E_i} 
\partial_t^\top, E_i) \nonumber.
\end{eqnarray}

\noindent where we have used that $(\nabla_X A)Y = \nabla_X (AY) - A(\nabla_X Y)$ 
for all $X, Y \in \mathfrak{X}(M)$. Combining now (\ref{nt}),
Codazzi equation 
$\overline{g}(\overline{\mathrm{R}}(X,Y) N, Z) = 
\overline{g}((\nabla_Y A) X, Z) - \overline{g}((\nabla_X A) Y, Z)$ (where 
$\overline{\mathrm{R}}$ denotes the curvature tensor of 
$I \times_f F$) and the fact that $H=0$ 
we are able to obtain from (\ref{lap3})

\begin{eqnarray}
\label{lap4}
\Delta \cosh \varphi &=& \sum_{i=1}^n\overline{g}(\overline{\mathrm{R}}(\partial_t^\top, E_i)N, E_i) + \sum_{i=1}^ng((\nabla_{\partial_t^\top} A) E_i, E_i) + \frac{f'(\tau)}{f(\tau)} g(A \partial_t^\top, \partial_t^\top) \\ \nonumber 
& & + \cosh \varphi \ |A|^2  - \frac{f''(\tau)}{f(\tau)} \cosh \varphi \sinh^2 \varphi  + 3 \frac{f'(\tau)^2}{f(\tau)^2} \cosh \varphi \sinh^2 \varphi \\ \nonumber
& & + \frac{f'(\tau)}{f(\tau)} g(A \partial_t^\top, \partial_t^\top) + n \frac{f'(\tau)^2}{f(\tau)^2} \cosh \varphi,
\end{eqnarray}

\noindent where $|A|^2 = \mathrm{trace}(A^2)$. Since 
covariant derivations commute with contractions, if we 
choose our local frame in $T_pM$ satisfying $\left(\nabla_{E_j} E_i \right)_p = 0$, we have

 \begin{eqnarray}
\label{lap5}
\Delta \cosh \varphi &=& - \overline{\mathrm{Ric}}(\partial_t^\top, N) + 2 \frac{f'(\tau)}{f(\tau)} g(A \partial_t^\top, \partial_t^\top)  \\ \nonumber 
& & + \cosh \varphi \ |A|^2 - \frac{f''(\tau)}{f(\tau)} \cosh \varphi \sinh^2 \varphi \\ \nonumber
& & + \ \frac{f'(\tau)^2}{f(\tau)^2} \cosh \varphi (n + 3 \sinh^2 \varphi),
\end{eqnarray}

\noindent where $\overline{\mathrm{Ric}}$ is the Ricci tensor of 
$I \times_f F$. If we write $N$ as $N=N^F-\overline{g}(N,\partial_t)\partial_t$, where
$N^F$ is the projection of $N$ on the fiber $F$, we obtain from 
\cite[Cor. 7.43]{O'N} 

\begin{equation}
\label{rict}
\overline{\mathrm{Ric}}(\partial_t,\partial_t)=-n \frac{f''(\tau)}{f(\tau)}
\end{equation}

\noindent and 

\begin{equation}
\label{ricNF}
\overline{\mathrm{Ric}}(N^F,N^F)= \mathrm{Ric}^F( N^F,N^F) + \sinh^2\varphi \left(\frac{f''(\tau)}{f(\tau)}+(n-1)\frac{f'(\tau)^2}{f(\tau)^2} \right),
\end{equation}

\noindent being $\mathrm{Ric}^F$ the Ricci tensor of $F$. Hence, from (\ref{rict}) and (\ref{ricNF}) we conclude

\begin{equation}
\label{ritn}
\overline{\mathrm{Ric}}(\partial_t^\top, N) = -\cosh \varphi \left( \mathrm{Ric}^F( N^F,N^F) - (n-1) (\log f)''(\tau) \sinh^2 \varphi \right).
\end{equation}

\noindent Introducing (\ref{ritn}) in (\ref{lap5}) we have

\begin{eqnarray}
\label{lap6}
\Delta \cosh \varphi &=& \cosh \varphi \left( \mathrm{Ric}^F( N^F,N^F) - (n-1) (\log f)''(\tau) \sinh^2 \varphi \right) \\ \nonumber 
& & + 2 \ \frac{f'(\tau)}{f(\tau)} g(A \partial_t^\top, \partial_t^\top) + \cosh \varphi \ |A|^2 \\ \nonumber 
& &  - \frac{f''(\tau)}{f(\tau)} \cosh \varphi \sinh^2 \varphi + \ \frac{f'(\tau)^2}{f(\tau)^2} \cosh \varphi (n + 3 \sinh^2 \varphi).
\end{eqnarray}

\noindent Our following key step is to use \eqref{nt} to compute 
$|\mathrm{Hess}(\tau)|^2$, having

\begin{eqnarray}
\label{he1}
|\mathrm{Hess}(\tau)|^2 &=& \sum_{i=1}^ng(\nabla_{E_i} \partial_t^\top, \nabla_{E_i} \partial_t^\top) = 2 \ \frac{f'(\tau)}{f(\tau)} \cosh \varphi \ g(A \partial_t^\top, \partial_t^\top) \\ \nonumber
& & + \cosh^2 \varphi \ |A|^2 + \frac{f'(\tau)^2}{f(\tau)^2} \left( n + 2 \sinh^2 \varphi + \sinh^ 4 \varphi \right) \\ \nonumber
\end{eqnarray}

\noindent Multiplying \eqref{lap6} by $\cosh \varphi$ 
and inserting \eqref{he1} provides 

\begin{eqnarray}
\label{clap0}
\cosh \varphi \ \Delta \cosh \varphi &=& \cosh^2 \varphi \left( \mathrm{Ric}^F( N^F,N^F) - (n-1) (\log f)''(\tau) \sinh^2 \varphi \right) \\ \nonumber
& &  + |\mathrm{Hess}(\tau)|^2  - \frac{f'(\tau)^2}{f(\tau)^2} \left( n + 2 \sinh^2 \varphi + \sinh^ 4 \varphi \right) \\ \nonumber 
& & - \frac{f''(\tau)}{f(\tau)} \cosh^2 \varphi \sinh^2 \varphi +  \frac{f'(\tau)^2}{f(\tau)^2} \cosh^2 \varphi (n + 3 \sinh^2 \varphi). \\ \nonumber
\end{eqnarray}

\noindent Summing up, \eqref{la1} and \eqref{clap0} yield 
(compare with \cite[Prop. 3.1]{LR}).

\begin{lema}
\label{lemon}
Let $\psi: M \longrightarrow I \times_f F$ be a  
maximal hypersurface in a GRW spacetime. Then, 

\begin{eqnarray}
\label{clap1}
\Delta \sinh^2 \varphi &=& 2 \cosh^2 \varphi \left( \mathrm{Ric}^F( N^F,N^F) - (n-1) (\log f)''(\tau) \sinh^2 \varphi \right) \\ \nonumber
& &  + 2 \ |\mathrm{Hess}(\tau)|^2  - 2 \ (\log f)''(\tau) \ (1 + \sinh^2 \varphi) \sinh^2 \varphi  \\ \nonumber 
& & + 2 \ \frac{f'(\tau)^2}{f(\tau)^2} \ (n + \sinh^2 \varphi) \sinh^2 \varphi 
+ 2 \ |\nabla \cosh \varphi|^2. \\ \nonumber
\end{eqnarray}
\end{lema}

Let us now introduce an energy condition that will be key to obtain our main results. Recall 
that a spacetime obeys the Null Convergence Condition (NCC) if 

$$ \overline{{\rm Ric}}(Z, Z) \geq 0,$$

\noindent for all lightlike vector $Z$. This energy condition holds in any 
spacetime satisfying Einstein's equation with a 
stress-energy tensor that obeys the weak energy condition and mathematically codifies 
gravity's attractive nature \cite[p. 95]{HE}. From \cite[Cor. 7.43]{O'N}, we can see that a 
GRW spacetime obeys the NCC if and only if

\begin{equation}
\label{ncc12}
\mathrm{Ric}^F(X,X) - (n-1) f(t)^2 (\log f)''(t) g_F(X, X) \geq 0 ,
\end{equation}

\noindent for all $X$ tangent to the fiber $F$. We are now in 
a position to state our main result, which can be regarded as 
an extension of  \cite[Thm. 8]{PRR2} and \cite[Thm. 10]{RRS3} 
to the stochastically complete case.

\begin{teor}
\label{teomon}

Let $\psi: M \longrightarrow I \times_f F$ be a stochastically complete 
maximal hypersurface in a GRW spacetime satisfying the NCC. If

\begin{equation}
\label{suplog}
\sup_M \ (\log f)''(\tau) < 0 , 
\end{equation}

\noindent and the hyperbolic angle is bounded, then $M$ is contained 
in a totally geodesic spacelike slice.
\end{teor}

\begin{proof}

Let $\sup_M (\log)''(\tau) = -\lambda < 0$, for some 
$\lambda >0$. Since the NCC is satisfied, from 
\eqref{clap1} we have 

\begin{equation}
\label{lasi}
\Delta \sinh^2 \varphi \geq 2 \lambda (1 + \sinh^2 \varphi) \sinh^2 \varphi.
\end{equation}

If $M$ is stochastically complete and $\sup_M \varphi \leq + \infty$, 
we obtain from \eqref{lasi} and Lemma \ref{lemastco} 
that $\varphi = 0$. Therefore, $M$ is 
contained in a spacelike slice. Moreover, since $M$ is maximal, this 
slice must be totally geodesic.
\end{proof}

\begin{rem}
\label{rema1}
\normalfont
Note that assumption \eqref{suplog} in Theorem \ref{teomon} cannot be 
omitted, being the spacelike hyperplanes in $\mathbb{L}^{n+1}$ (even 
if we remove a compact subset from them \cite[Cor. 6.5]{Gr})
examples of stochastically complete maximal hypersurfaces 
with bounded hyperbolic angle in a GRW 
spacetime that satisfies the NCC which are not contained in a 
spacelike slice. 
\end{rem}

\begin{rem}
\normalfont
The bound on the hyperbolic angle has an interesting physical 
interpretation. Indeed, on any point of a spacelike hypersurface 
$p \in M$ there are 
two distinguished instantaneous 
observers: $N (p)$ and $\partial_t (p)$. Using the decomposition 
$N = \cosh \varphi \ \partial_t  + N^F $ we see that 
the boundedness of the hyperbolic angle implies that 
the Newtonian speed of $N$ measured by $\partial_t$, 
defined by $v:= \tanh \varphi$ does not approach the speed of light 
in vacuum \cite[p. 45]{SW}.
\end{rem}

As a direct consequence of Theorem \ref{teomon} we can prove the 
next rigidity result in the open region of anti-de Sitter spacetime 
modeled by  
$(-\pi/2, \pi/2) \times_{\cos t} \mathbb{H}^n$ (see \cite[Sec. 4]{Mo}).

\begin{coro}
\label{corads}
Let $\psi: M \longrightarrow (-\pi/2, \pi/2) \times_{\cos t} \mathbb{H}^n$ 
be a stochastically complete maximal hypersurface with bounded hyperbolic 
angle. Then, $M$ is contained in the totally geodesic spacelike slice 
$\{0\} \times \mathbb{H}^n$.
\end{coro}

Furthermore, we can obtain certain non existence results for 
stochastically complete maximal 
hypersurfaces with bounded hyperbolic angle in a wide class 
of GRW spacetimes.

\begin{teor}
\label{teononex}
There are no stochastically complete maximal hypersurfaces with bounded 
hyperbolic angle in a GRW specetime that satisfies the NCC 
with $(\log f)'' \leq~ 0$ and $\inf (f'/f)^2 >~ 0$. 
\end{teor}

\begin{proof}

Let us assume the existence of such a hypersurface. Using these 
assumptions in \eqref{clap1} and calling 
$\inf (f'/f)^2 = \alpha > 0$ we obtain

\begin{equation}
\label{lasi2}
\Delta \sinh^2 \varphi \geq 2 \alpha (n + \sinh^2 \varphi) \sinh^2 \varphi.
\end{equation}

If $M$ is stochastically complete and $\sup_M \varphi \leq + \infty$, 
we obtain from \eqref{lasi2} and Lemma \ref{lemastco} that $M$ should be 
contained in a totally geodesic spacelike slice, 
reaching a contradiction.
\end{proof}

Theorem \ref{teononex} yields interesting consequences for 
the steady state spacetime proposed by Bondi, Gold and Hoyle 
\cite[p. 126]{HE}, 
which can be described by the Robertson-Walker 
spacetime $\mathbb{R} \times_{e^t} \mathbb{R}^n$.

\begin{coro}
\label{coroste}
There are no stochastically complete maximal hypersurfaces with 
bounded hyperbolic angle in the steady state spacetime 
$\mathbb{R} \times_{e^t} \mathbb{R}^n$.
\end{coro}

Moreover, for product spacetimes we have. 

\begin{teor}
\label{teopro}
Let $\psi: M \longrightarrow I \times F$ be a stochastically complete 
maximal hypersurface in a product spacetime whose fiber's 
Ricci curvature satisfies $\mathrm{Ric}^F \geq \beta g_F $ for 
$\beta >0$. If 
the hyperbolic angle of $M$ is bounded, then $M$ 
is contained in a spacelike slice.
\end{teor}

\begin{proof}
If the above bound on the Ricci curvature of the fiber 
holds, \eqref{clap1} yields

\begin{equation}
\label{lassi}
\Delta \sinh^2 \varphi \geq 2 \beta (1 + \sinh^2 \varphi) \sinh^2 \varphi .
\end{equation}

If $M$ is stochastically complete and $\sup_M \varphi \leq + \infty$, 
\eqref{lassi} and Lemma \ref{lemastco} allow us to conclude that 
$M$ is contained in a spacelike slice.
\end{proof}

\begin{rem}
\normalfont
For product spacetimes the positivity of the fiber's Ricci curvature 
is key in order to obtain these uniqueness results. Indeed, for the case 
of flat fiber the spacelike hyperplanes in $\mathbb{L}^{n+1}$ previously 
mentioned in Remark \ref{rema1} show how these uniqueness result do not 
hold. 

In addition, if the fiber has negative Ricci curvature, we can consider 
in the Lorentzian product spacetime 
$\mathbb{R}\times \mathbb{H}^2$ the graph of the function 
$u(x_1, x_2) = \frac{1}{3} \log(x_1^2 + x_2^2)$, $(x_1, x_2) \in \mathbb{H}^2$, which defines a complete maximal spacelike 
hypersurface with bounded hyperbolic angle (see \cite{Alb}).

Moreover, choosing a local orthonormal reference 
frame $\{E_1, E_2\}$ on the surface and using the Gauss equation 
the Ricci curvature of this maximal surface satisfies 

\begin{equation}
\label{rimi1}
\mathrm{Ric}(X, X) = \sum_{i=1}^2 \overline{g}(\overline{\mathrm{R}}(X, E_i) E_i, X) + |A X|^2,
\end{equation}
 
\noindent for all $X \in \mathfrak{X}(\Sigma_u)$, where 
$\overline{\mathrm{R}}$ denotes the curvature tensor of $\mathbb{R} \times \mathbb{H}^2$. Using 
\cite[Prop. 7.42]{O'N} we obtain after several computations

\begin{equation}
\label{rimi2}
\mathrm{Ric}(X, X) = - \cosh^2 \varphi \ |X|^2 + |A X|^2.
\end{equation}

Therefore, since the hyperbolic angle is bounded, the Ricci curvature 
is bounded from below. This fact together with the completeness ensures 
that this maximal graph is stochastically complete 
due to Lemma \ref{teooy}.
\end{rem}

As a direct consequence of Theorem \ref{teopro} we obtain the following 
result in Einstein's static 
universe $\mathbb{R} \times \mathbb{S}^n$ \cite[p. 121]{HE}.

\begin{coro}
\label{corosta}
Let $\psi: M \longrightarrow \mathbb{R} \times \mathbb{S}^n$ be a 
stochastically complete maximal hypersurface with bounded hyperbolic 
angle in Einstein's static universe. Then, $M$ is contained in 
a spacelike slice $\{t_0\} \times \mathbb{S}^n$, $t_0 \in \mathbb{R}$.
\end{coro}

\subsection{Lorentz-Minkowski and de Sitter spacetimes}
\label{suselm}

Notice that our previous results do 
not provide any information 
about stochastically complete maximal hypersurfaces in two distinguished 
Robertson-Walker models. Namely,  
Lorentz-Minkowski spacetime $\mathbb{L}^{n+1}$ and de Sitter 
spacetime $\mathbb{S}_1^{n+1}$ (which can be regarded as 
the Robertson-Walker 
$\mathbb{R} \times_{\cosh(t)} \mathbb{S}^n$). These two 
models also belong to the relevant 
class of Lorentzian spacetimes with constant 
sectional curvature. Indeed, 
using the following well known 
Simons-type formula for maximal hypersurfaces in 
Lorentzian space forms (see, for example, \cite[Lemma 9.7]{AMR}) 
we can obtain a new rigidity result for 
stochastically complete maximal hypersurfaces in these 
ambient spacetimes. 

\begin{lema}
Let $\psi: M^n \longrightarrow \overline{M}^{n+1}$ be 
a maximal hypersurface in 
a spacetime $\overline{M}$ of constant sectional 
curvature $\overline{c}$. Then

\begin{equation}
\label{sim}
\frac{1}{2} \Delta |A|^2 = |\nabla A |^2 + (n \overline{c} + |A|^2 ) |A|^2,
\end{equation}

\noindent where $A$ is the shape operator of the maximal hypersurface.
\end{lema}

This lemma allows us to obtain the following result 

\begin{teor}
\label{teosf}
Let $\psi: M \longrightarrow \overline{M}$ be 
a stochastically complete maximal hypersurface in 
a spacetime $\overline{M}$ of constant sectional 
curvature $\overline{c} \geq 0$. If $\sup_M |A|^2 < + \infty$, then $M$ 
is totally geodesic.
\end{teor}

\begin{proof}
If $\overline{c} \geq 0$, we have from \eqref{sim}

\begin{equation}
\label{lapsf}
\Delta |A|^2 \geq 2 |A|^4.
\end{equation}

If $M$ is stochastically complete and $|A|^2$ is bounded from above, 
Lemma \ref{lemastco} enable us to conclude 
that $M$ must be totally geodesic.
\end{proof}

\begin{rem}
\normalfont
Note that Theorem \ref{teosf} extends the classical 
Calabi-Bernstein 
theorem to the stochastically complete case 
in $\mathbb{L}^{n+1}$ and 
$\mathbb{S}_1^{n+1}$ by means of a bound on the shape 
operator (compare with \cite[Thm. C]{CPRi}). In fact, this bound 
holds on certain 
spacelike embeddings of the sphere into de Sitter 
spacetime \cite[Thm. 1]{BKL}. Moreover, if we assume the 
completeness of the 
maximal hypersurface we can omit the bound on the 
shape operator since 
the classical Omori-Yau maximum principle 
holds on every complete 
maximal hypersurface in a Lorentzian spacetime of 
constant sectional 
curvature (see, for instance, \cite[Prop. 1]{Rm}), which was the 
key idea used in Cheng and Yau's proof of the 
Calabi-Bernstein theorem in $\mathbb{L}^{n+1}$ \cite{CY}.
\end{rem}

\section{A Calabi-Bernstein type result}
\label{secb}

Let $(F, g_F)$ be an $n(\geq 2)$-dimensional Riemannian manifold and 
let $f: I \longrightarrow \mathbb{R}^+$ be a smooth 
function. Given the GRW spacetime $I \times_f F$, consider 
the graph 

$$\Sigma_u = \{(u(p), p) : p \in \Omega \},$$

\noindent where $\Omega \subseteq F$, $u \in C^\infty(\Omega)$ 
and $u(\Omega)\subset I$. The Lorentzian metric on $I \times_f F$ induces, 
via the graph $\Sigma_u$, a metric on 
$\Omega$ given by

\begin{equation}
\label{megra}
 g_u = -du^2 + f(u)^2 g_F.
\end{equation}

\noindent This metric $g_u$ is Riemannian (and thus, $\Sigma_u$ is spacelike) 
if and only if

\begin{equation}
\label{graespa}
|Du| < f(u).
\end{equation}

\noindent In this case, the future pointing unit normal vector 
field on $\Sigma_u$ is

\begin{equation}
\label{normalvec}
N = \frac{1}{f(u) \sqrt{f(u)^2 - |Du|^2}} \left( f(u)^2 \partial_t + Du \right).
\end{equation}

\noindent The spacelike graph is said to be 
entire when $\Omega = F$. Now, \cite[Prop. 7.35]{O'N} allows us 
to obtain that the shape operator associated to $N$ is given by

\begin{eqnarray}
\label{shag}
A X &=& \frac{-f(u)}{\sqrt{f(u)^2 - |Du|^2}} \left( \frac{f'(u)}{f(u)} X - \frac{f'(u) \ g_F(Du, X)}{f(u) \left( f(u)^2 - |Du|^2 \right)} Du \right. \nonumber \\ 
\vspace*{2mm}
&& \left. + \frac{g_F(D_X Du, Du)}{f(u)^2 \left( f(u)^2 - |Du|^2 \right)} Du + \frac{1}{f(u)^2} D_X Du \right)
\end{eqnarray}

\noindent for all $X \in \mathfrak{X}(\Omega)$. Computing the trace of (\ref{shag}) 
yields that the mean curvature function of a spacelike graph associated to $N$ is

\begin{equation}
\label{curvmed}
H =  \mathrm{div} \left( \frac{Du}{n f(u) \sqrt{f(u)^2 - |Du|^2}} \right) 
 +\frac{f'(u)}{n \sqrt{f(u)^2 - |Du|^2}} \left( n + \frac{|Du|^2}{f(u)^2} \right),
\end{equation}

\noindent where $\mathrm{div}$ represents the divergence operator in 
$(F, g_F)$. We can use Theorem \ref{teomon} to prove a 
new uniqueness result for 
the following uniformly elliptic PDE, which extends 
\cite[Thm. 22]{ARR} and \cite[Thm. 5.11]{RRS1} to a 
nonparabolic Riemannian manifold $F$.

\begin{teor}
\label{teocb}
Let $f: I \longrightarrow \mathbb{R}^+$ be a smooth function such 
that $\inf f > 0$ and 
$\sup \ (\log f)'' < ~0$. Then, the only entire solutions 
to the equation

\begin{gather*}
\label{calb} 
\mathrm{div} \left( \frac{Du}{f(u) \sqrt{f(u)^2 - |Du|^2}} \right) 
 = - \frac{f'(u)}{\sqrt{f(u)^2 - |Du|^2}} \left( n + \frac{|Du|^2}{f(u)^2} \right), \tag{E.1} \\[2ex]
|Du| < \lambda f(u), \ \ 0 < \lambda < 1, \tag{E.2}
\end{gather*}

\noindent on 
a complete Riemannian manifold $F$ with non negative sectional 
curvature are the 
constant functions $u = t_0$, with $t_0 \in I$ such 
that $f'(t_0) = 0$.
\end{teor}

\begin{proof}

Let us begin by observing that (E.2) implies that 
the hyperbolic angle of the graph $\Sigma_u$ is bounded by 

\begin{equation}
\label{chb}
\cosh \varphi < \frac{1}{\sqrt{1 - \lambda^2}}.
\end{equation}

Furthermore, from Schwarz inequality we can deduce

\begin{equation}
\label{guu}
g_u(v, v) \geq |Du|^2 g_u(v, v) + f(u)^2 g_F( d \pi_F (v), d \pi_F (v)), \ \mathrm{for \ all} \ v \in T \Sigma_u.
\end{equation}

Therefore,

\begin{equation}
\label{guu2}
g_u(v, v) \geq \frac{f(u)^2}{\cosh^2 \varphi} g_F( d \pi_F (v), d \pi_F (v)).
\end{equation}

Denoting by $\mathcal{L}_F(\gamma)$ and $\mathcal{L}_u(\gamma)$ 
the length 
of a smooth curve $\gamma$ on $F$ with respect to the 
metrics $g_F$ and $g_u$, respectively, using (\ref{guu2}) 
and (\ref{chb}) we obtain

\begin{equation}
\label{lu3}
\mathcal{L}_u(\gamma) \geq (1 - \lambda^2) (\inf f(u)^2) \mathcal{L}_F(\gamma).
\end{equation} 

Thus, the completeness of $(F, g_F)$ and the fact that $\inf f > 0$ 
ensures that the metric $g_u$ is also complete. Once the completeness 
of the graph is guaranteed, we also 
have from the Gauss equation

\begin{equation}
\label{gausseq}
g(\mathrm{R}(X, Y) Z, W) = \overline{g}(\overline{\mathrm{R}}(X, Y) Z, W) - g(A Y, Z) g(A X, W) + g(A X, Z) g(A Y, W),
\end{equation}

\noindent for $X, Y, Z, W \in \mathfrak{X}(M)$, where $\mathrm{R}$ and $\overline{\mathrm{R}}$ denote the curvature tensors of $\Sigma_u$ and 
$I \times_f F$ respectively. Choosing a local orthonormal reference 
frame $\{E_1, \dots , E_n\}$ around a point $p \in \Sigma_u$ 
we deduce that the Ricci curvature of the maximal graph verifies

\begin{equation}
\label{rim1}
\mathrm{Ric}(X, X) \geq \sum_{i=1}^n \overline{g}(\overline{\mathrm{R}}(X, E_i) E_i, X),
\end{equation}
 
\noindent for all $X \in \mathfrak{X}(\Sigma_u)$. From \cite[Prop. 7.42]{O'N} we obtain

\begin{eqnarray}
\label{rim2}
\sum_{i=1}^n \overline{g}(\overline{\mathrm{R}}(X, E_i) E_i, X) &=& f(u)^2 \sum_{i=1}^n g_F(\mathrm{R^F}(X^F, E_i^F) E_i^F, X^F) \\ \nonumber
& & + (n-1) \frac{f'(u)^2}{f(u)^2} |X|^2 - (\log f)''(u) \sinh^2 \varphi |X|^2 \\ \nonumber
& &  -(n-2)(\log f)''(u) \overline{g}(X, \nabla u)^2,
\end{eqnarray}

\noindent where $\mathrm{R^F}$ is the curvature tensor of 
$F$ and $X^F$ and $E_i^F$ denote
the projections of $X$ and $E_i$ on the fiber. If 
$\sup \ (\log f)'' < ~0$ and 
the sectional curvature of the fiber is non negative, we can clearly 
see combining (\ref{rim1}) and (\ref{rim2}) that 
the Ricci curvature of the maximal graph is bounded from below. Since  
it is also complete, Lemma \ref{teooy} ensures that 
the Omori-Yau maximum principle holds 
on $\Sigma_u$. Consequently, the spacelike graph is a stochastically 
complete maximal hypersurface in this GRW spacetime that satisfies 
the NCC, so we can 
call Theorem \ref{teomon} to conclude the proof.
\end{proof}

\section*{Acknowledgements} 

The second author is partially supported by Spanish MICINN 
project PID2020-116126GB-I00.

\end{document}